\documentclass[11pt,a4paper,reqno]{amsart}
\usepackage[british]{babel}

\usepackage{a4wide}
\usepackage{setspace}
\usepackage{amssymb}
\usepackage{amsmath}
\usepackage{amsthm}
\usepackage{array}
\usepackage{enumitem}
\usepackage[colorlinks,citecolor=blue,urlcolor=black,linkcolor=red,linktocpage]{hyperref}
\usepackage{mathrsfs}
\usepackage{mathtools}

\usepackage[utf8]{inputenc}

\newtheorem{thm}{Theorem}[section]

\newtheorem{cor}[thm]{Corollary}
\newtheorem{lem}[thm]{Lemma}

\newtheorem{prop}[thm]{Proposition}

\theoremstyle{definition}

\newtheorem{definition}[thm]{Definition}

\newtheorem*{quest}{Question}

\renewcommand{\epsilon}{\varepsilon}
\renewcommand{\phi}{\varphi}
\newcommand{\defeq}{\mathrel{\mathop:}=}

\DeclareMathOperator{\rad}{rad}
\DeclareMathOperator{\soc}{soc}
\DeclareMathOperator{\tp}{top}
\DeclareMathOperator{\Hom}{Hom}

\newcommand{\ol}{\overline}

\begin{document}


\setlist{noitemsep}

\author{Friedrich Martin Schneider}
\thanks{F.M.S.\ acknowledges funding of the Excellence Initiative by the German Federal and State Governments.}
\address{F.M.S., Institut f\"ur Algebra, TU Dresden, 01062 Dresden, Germany}
\email{martin.schneider@tu-dresden.de}

\author{Jens Zumbr\"agel}
\address{J.Z., Faculty of Computer Science \& Mathematics, Universit\"at Passau, 94032 Passau, Germany}
\email{jens.zumbraegel@uni-passau.de}

\title{MacWilliams' extension theorem for infinite rings}
\date{\today}

\keywords{Quasi-Frobenius ring; Code equivalence}

\begin{abstract}
	Finite Frobenius rings have been characterized as precisely those finite rings satisfying the MacWilliams extension property, by work of Wood.  In the present note we offer a generalization of this remarkable result to the realm of Artinian rings.  Namely, we prove that a left Artinian ring has the left MacWilliams property if and only if it is left pseudo-injective and its finitary left socle embeds into the semisimple quotient.  Providing a topological perspective on the MacWilliams property, we also show that the finitary left socle of a left Artinian ring embeds into the semisimple quotient if and only if it admits a finitarily left torsion-free character, if and only if the Pontryagin dual of the regular left module is almost monothetic.  In conclusion, an Artinian ring has the MacWilliams property if and only if it is finitarily Frobenius, i.e., it is quasi-Frobenius and its finitary socle embeds into the semisimple quotient.
\end{abstract}


\maketitle



\section*{Introduction}

Quasi-Frobenius rings, i.e., rings which are Artinian and self-injective, belong to those classical Artinian rings which have been a driving force in the development of modern ring and module theory.  Introduced by Nakayama~\cite{nakayama}, one of their many equivalent characterizations identifies them as those rings~$R$ for which the $\Hom( -, R )$ functor provides a duality between its finitely generated left modules and its finitely generated right modules.  They also appear as the smallest categorical generalization of group rings of finite groups.

Within the class of quasi-Frobenius rings, those rings~$R$ for which the socle $\soc R$ is isomorphic, as one-sided module, to the semisimple quotient ring $R /\! \rad R$, are the Frobenius rings.  They emerge naturally as a generalization of Frobenius algebras, i.e., finite-dimensional algebras over a field that admit a non-degenerate balanced bilinear pairing.

In more recent years, with the advent of ring-linear coding theory (see, for example, \cite{ring-codes}), the interest in finite ring theory has increased.  One of the striking results in this regard is the characterization, due to Wood~\cite{wood1, wood2}, of finite Frobenius rings as precisely those rings~$R$ which satisfy the following MacWilliams extension property: every Hamming weight preserving isomorphism between left submodules of $R^n$ extends to a monomial transformation, i.e., is of the form $(x_i) \mapsto ( x_{\sigma i} u_i )$ for a permutation $\sigma \in S_n$ and invertible ring elements $u_i \in R$.  This property has been established by MacWilliams~\cite{macwilliams} in the case of finite fields and consolidates the notion of code equivalence.  A further remarkable result is the observation by Honold~\cite{honold} that in the finite case Frobenius rings can be characterized as those rings~$R$ satisfying a one-sided condition $\soc R \cong R /\! \rad R$, even without assuming the ring to be quasi-Frobenius.

In the present note we offer a generalization of Wood's characterization of finite Frobenius rings as rings satisfying the MacWilliams extension property to the realm of general infinite (Artinian) rings.  We remark that the case of infinite Artin algebras has recently been treated by Iovanov~\cite{iovanov}.

While each (two-sided) MacWilliams ring is necessarily quasi-Frobenius, it turns out that the Frobenius property is in general too strong to be deduced from satisfying MacWilliams' extension theorem.  We therefore weaken the Frobenius property to a criterion which we call \emph{finitarily Frobenius}, which merely requires that the finitary socle embeds into the semisimple quotient (either as left or right module).  Clearly, every Frobenius ring is also finitarily Frobenius.  In our main result, Theorem~\ref{theorem:main}, we show that a left Artinian ring satisfies the MacWilliams extension property for left modules if and only if it is left pseudo-injective and the finitary left socle embeds into the semisimple quotient.  It follows that an Artinian ring is finitarily Frobenius if and only if it satisfies MacWilliams' extension property.

Our approach is based on the description of Frobenius rings in terms of generating characters, an idea developed by Wood~\cite{wood1} and adapted by Iovanov~\cite{iovanov}.  In fact, our proof method relies on the existence of certain torsion-free characters on finitarily Frobenius rings, as well as results on Pontryagin duality of discrete and compact abelian groups.  Along the way, we show for a left Artinian ring that the finitary left socle embeds into the semisimple quotient if and only if it admits a finitarily left torsion-free character, if and only if the Pontryagin dual of the regular left module is almost monothetic.

\section{Frobenius rings and generalizations}\label{section:frobenius.rings}

We compile in this section a few notions from ring and module theory as needed in the present context and introduce the notion of a finitarily Frobenius ring.  For a comprehensive account on the classical theory we refer to~\cite{lam, anderson-fuller}, see also~\cite{wood1, honold}.

In the following, the term \emph{ring} will always mean unital ring.  Recall that a ring is said to be \emph{quasi-Frobenius} if it is left Artinian and left self-injective, i.e., injective as a left module.  As it turns out, the properties left Artinian and left self-injective can each be replaced by its right counterparts, and the Artinian property by Noetherian (cf.~\cite[Sec.~15]{lam}).

We shall, for a (left or right) module~$M$, denote by $\soc M$ the sum of all its minimal submodules, by $\rad M$ the intersection of all maximal ones and by $\tp M \defeq M /\! \rad M$ its “top quotient”.  Accordingly, we denote by $\rad R$ the \emph{Jacobson radical} of a ring~$R$, i.e., the intersection of all maximal left (right) ideals; also, let $\soc({}_R R)$ be its \emph{left socle}, i.e., the sum of all minimal left ideals, and let $\soc(R_R)$ be its analogously defined \emph{right socle}.  A crucial notion for the present note is the Frobenius property.

\begin{definition}\label{def:frob} A ring~$R$ is called \emph{Frobenius} if it is quasi-Frobenius and satisfies \[ \text{(i)} \ \soc({}_R R) \cong {}_R (R /\! \rad R) \quad \text{ and/or } \quad \text{(ii)} \ \soc(R_R) \cong (R /\! \rad R)_R \,. \] \end{definition}

For quasi-Frobenius rings the conditions (i) and (ii) are actually equivalent.  Indeed, it is worthwhile to recall how the properties of quasi-Frobenius and Frobenius may be expressed with respect to the principal decomposition, as we outline briefly below (for details, see~\cite[Sec.~16]{lam} or~\cite[Sec.~31]{anderson-fuller}).  Let~$R$ be a left or right Artinian ring and let $S \defeq R /\! \rad R$ be its semisimple quotient.  Then there is a list of orthogonal primitive idempotents $e_1, \dots, e_n \in R$ such that \[ R = R e_1 \oplus \ldots \oplus R e_n \quad \text{ and } \quad R = e_1 R \oplus \ldots \oplus e_n R \] are direct sums of indecomposable left and right modules, respectively, and, letting $\ol e_i \defeq e_i + \rad R \in S$ for all $i \in \{ 1, \dots, n \}$, one has decompositions \[ S = S \ol e_1 \oplus \ldots \oplus S \ol e_n \quad \text{ and } \quad S = \ol e_1 S \oplus \ldots \oplus \ol e_n S \] into simple left and right modules, respectively.  For all $i, j \in \{ 1, \dots, n \}$ there holds \[ S \ol e_i \cong S \ol e_j \ \Longleftrightarrow\  R e_i \cong R e_j \ \Longleftrightarrow\ e_i R \cong e_j R \ \Longleftrightarrow\  \ol e_i S \cong \ol e_j S \,, \] and we may assume that $R e_1, \dots, R e_m$ (for some $m \le n$) form a complete set of non-isomorphic representatives for all $R e_i$.  Then $S \ol e_1, \dots, S \ol e_m$ (and $\ol e_1 S, \ldots, \ol e_m S$) form an irredundant set of representatives for all simple left (right) modules.  We shall refer to $e_1, \dots, e_m$ as a \emph{basic set} of idempotents for the ring~$R$.  It is easy to see that $\tp(R e_i) \cong S \ol e_i$ and $\tp(e_i R) \cong \ol e_i S$ (considered as $R$-modules), in particular, the former are simple.

Now if the ring~$R$ is quasi-Frobenius then each of $\soc (R e_i)$ and $\soc (e_i R)$ is also simple.  In fact, the following characterization is valid (see, e.g., \cite[Cor.~31.4]{anderson-fuller}), which actually corresponds to Nakayama's original definition of quasi-Frobenius rings~\cite{nakayama}.

\begin{thm}\label{theorem:nakayama} Let~$R$ be a left or right Artinian ring with a basic set of idempotents $e_1, \dots, e_m$.  Then the ring~$R$ is quasi-Frobenius if and only if there is a permutation $\pi \in S_m$ such that \[ \soc(R e_i) \cong \tp(R e_{\pi(i)}) \quad \text{ and } \quad \soc(e_{\pi(i)} R) \cong \tp(e_i R) \,. \] \end{thm}

The permutation $\pi \in S_m$ in Theorem~\ref{theorem:nakayama} is referred to as the \emph{Nakayama permutation}.  Notice that for any fixed~$j$ the number~$\mu_j$ of indecomposables $R e_i$ isomorphic to $R e_j$ equals the number of simples $\tp(R e_i)$ isomorphic to $\tp(R e_j)$, and coincides with its right counterpart.  Hence, for a quasi-Frobenius ring~$R$, Theorem~\ref{theorem:nakayama} yields that \[ \soc({}_R R) = \textstyle\bigoplus\limits_{i=1}^n \soc(R e_i) \cong \textstyle\bigoplus\limits_{i=1}^n \tp(R e_i) = \tp({}_R R) \] if and only if $\mu_{\pi(i)} = \mu_i$ for all $i \in \{ 1, \dots, n \}$, which in turn is equivalent to $\soc(R_R) \cong \tp(R_R)$.  This shows the equivalence of condition (i) and (ii) of Definition~\ref{def:frob} for quasi-Frobenius rings.  (On the other hand, any Artinian ring~$R$ satisfying both $\soc({}_R R) \cong \tp({}_R R)$ and $\soc(R_R) \cong \tp(R_R)$ is necessarily quasi-Frobenius.)

We are going to introduce a finitary version of the Frobenius property. Given a ring $R$, we define its \emph{finitary left socle} $\soc^{\ast}({}_R R)$ to be the sum of all finite minimal left ideals of~$R$, and its \emph{finitary right socle} $\soc^{\ast}(R_R)$ as the sum of all finite minimal right ideals of~$R$.

\begin{prop}\label{proposition:finitary.frobenius} Let~$R$ be a quasi-Frobenius ring. Then $\soc^{\ast}({}_R R)$ embeds into ${}_R (R /\! \rad R)$ if and only if $\soc^{\ast}(R_R)$ embeds into $(R /\! \rad R)_R$. \end{prop}

\begin{proof} Let $e_1, \dots, e_m$ be a basic set of idempotents for the ring~$R$. First we observe that $\soc(R e_i)$ is finite if and only if $\tp(e_i R)$ is finite.  Indeed, since~$R$ is quasi-Frobenius we have $\Hom(\soc(R e_i), R) \cong \tp(e_i R)$ and $\Hom(\tp(e_i R), R) \cong \soc(R e_i)$ (see~\cite[Cor.~16.6]{lam} or~\cite[Cor.~2.5]{wood1}).  Furthermore, if~$T$ is any finite simple module, then $\Hom(T, R)$ is finite, since every homomorphism $T \to R$ maps into the finite set $\soc^{\ast}(R)$.

Next it is easy to see that $\tp(R e_i) \cong S \ol e_i$ is finite if and only if $\ol e_i S \cong \tp(e_i R)$ is finite; in fact they are isomorphic to the standard column and row modules of the same matrix ring in the Artin-Wedderburn decomposition of $S \defeq R /\! \rad R$.  This shows that $\soc(R e_i)$, $\tp(R e_i)$, $\soc(e_i R)$, $\tp(e_i R)$ are, for each~$i$, simultaneously either finite or infinite.

Now denoting by~$F$ the set of all~$i$ such that $\soc(R e_i)$ is finite, we see from Theorem~\ref{theorem:nakayama} that the Nakayama permutation~$\pi$ preserves the set~$F$, and from the subsequent discussion that $\soc^{\ast}({}_R R)$ embeds into $\tp({}_R R)$ if and only if $\mu_{\pi(i)} = \mu_i$ for all $i \in F$, which holds if and only if $\soc^{\ast}(R_R)$ embeds into $\tp(R_R)$. \end{proof}

In view of Proposition~\ref{proposition:finitary.frobenius}, we record the following definition.

\begin{definition}\label{def:finitely.frobenius} A ring~$R$ is called \emph{finitarily Frobenius} if it is quasi-Frobenius and there holds one (thus each) of the following equivalent conditions: \begin{enumerate}[label=(\roman*)] \item $\soc^*({}_R R)$ embeds into ${}_R (R /\! \rad R)$. \item $\soc^{\ast}(R_R)$ embeds into $(R /\! \rad R)_R$. \end{enumerate} \end{definition}

The following notion will also be relevant for the MacWilliams extension property.  A left module ${}_R M$ is said to be \emph{pseudo-injective} if for every submodule~$N$ of~$M$ and any injective homomorphism $f \colon N \to M$ there is an endomorphism $g \colon M \to M$ with $g|_N = f$.  Accordingly, a ring~$R$ is called \emph{left (right) pseudo-injective} if for every left (right) ideal~$I$, each injective homomorphism $I \to R$ is given by a right (left) multiplication by an element of~$R$.  Dinh and L\'opez-Permouth have shown~\cite[Prop.~3.2]{DinhLopezA} that a finite ring is left pseudo-injective if and only if it satisfies the MacWilliams property for codes of length one.

Clearly, every injective module is pseudo-injective and every quasi-Frobenius ring is left and right pseudo-injective.  There is considerable interest in such weaker forms of self-injectivity, one motivation being to discuss more general assumptions on rings that imply quasi-Frobenius.  In particular, a ring~$R$ is termed \emph{left (right) min-injective} if for every minimal left (right) ideal~$I$ each homomorphism $I \to R$ is given by right (left) multiplication; see, e.g., \cite{harada, nicholson-yousif}.  Note that left (right) pseudo-injectivity implies left (right) min-injectivity for rings.  Pseudo-injective modules gained attention more recently, as they are characterized as modules that are invariant under automorphisms of the injective envelope~\cite{ErSinghSrivastava}.

Let us record the following result.

\begin{prop}\label{prop:quasi-Frobenius.pseudo-injective} An Artinian ring is quasi-Frobenius if and only if it is both left and right pseudo-injective. \end{prop}

\begin{proof} Since pseudo-injectivity implies min-injectivity, the result is a direct consequence of~\cite[Thm.~13]{harada}; see also~\cite[Thm.~3.12]{iovanov}. \end{proof}

The following useful observation is implicit in~\cite[Cor.~3.5]{iovanov}, and for the reader's convenience we include a direct argument based on work of Bass~\cite{bass}, along the lines of \cite[Prop.~5.1]{wood1} and \cite[Lem.~3.3]{iovanov}.

\begin{lem}\label{lemma:bass} Let~$R$ be a left or right Artinian ring which is left pseudo-injective.  If~${}_R M$ is a left $R$-module and $g, h \colon M \to R$ are homomorphisms such that $\ker g = \ker h$, then there exists a unit $u \in R$ such that $h(x) = g(x) u$ for all $x \in M$. \end{lem}

\begin{proof} Consider the induced injective maps $\widetilde g, \smash{\widetilde h} \colon M / N \to R$, where $N \defeq \ker g = \ker h$.  Letting $I \defeq \mathrm{im}\, \widetilde g$ we have an injective homomorphism $f \defeq \smash{\widetilde h} \circ \widetilde g^{-1} \colon I \to R$, thus by pseudo-injectivity there exists $a \in R$ with $f(z) = z a$ for all $z \in I$, which implies $h(x) = g(x) a$ for all $x \in M$.  Similarly, we find $b \in R$ such that $g(x) = h(x) b$ for all $x \in M$.  Now since $R = a b R + (1 - a b) R \subseteq a R + (1 - a b) R$ and $R /\! \rad R$ is semisimple, it follows from~\cite[Lem.~6.4]{bass} that there is a unit $u \in R$ such that $u = a + (1 - a b) r$ for some $r \in R$.  Thus $g(x) u = g(x) a + g(x) (1 - a b) r = g(x) a = h(x)$ for all $x \in M$, as desired. \end{proof}

\section{Torsion-free characters}\label{section:characters}

In this section we show that any Frobenius ring admits a left (resp., right) torsion-free character and that, similarly, every finitarily Frobenius ring admits a finitarily left (resp., right) torsion-free character, cf.~Definition~\ref{definition:torsion-free}.  Let us start with a fairly general setting.

\begin{definition}\label{definition:torsion-free} Let~$R$ be a ring and let~$E$ be an abelian group.  A homomorphism $\chi \colon {R \to E}$ is called \emph{left torsion-free} (resp., \emph{right torsion-free}) if the subgroup $\ker \chi$ contains no nonzero left (resp., right) ideals.  The homomorphism~$\chi$ is called \emph{torsion-free} if it is both left- and right torsion-free.  Furthermore, $\chi \colon {R \to E}$ is said to be \emph{finitarily left torsion-free} (resp., \emph{finitarily right torsion-free}) if the subgroup $\ker \chi$ contains no nonzero finite left (resp., right) ideals; and~$\chi$ is called \emph{finitarily torsion-free} if it is both finitarliy left- and right torsion-free. \end{definition}

Our construction of torsion-free characters on Frobenius rings depends upon the celebrated Artin-Wedderburn theorem.  We start with the case of division rings.

\begin{lem}\label{lemma:torsion-free.division} Every division ring~$D$ admits a torsion-free homomorphism into $\mathbb Q / \mathbb Z$. \end{lem}

\begin{proof} The $\mathbb Z$-module $\mathbb Q / \mathbb Z$ has the cogenerator property, i.e., for any abelian group~$X$ and every nonzero $x \in X$ there is a homomorphism $f \colon X \to \mathbb Q / \mathbb Z$ with $f(x) \ne 0$ (cf.~\cite[Lem.~4.7]{lam}).  In particular, there exists a nonzero homomorphism $\chi \colon D \to \mathbb Q / \mathbb Z$, which clearly must be torsion-free as~$D$ does not admit any non-trivial left or right ideals. \end{proof}

Given a ring $R$ and some integer $n \geq 1$, we consider the \emph{matrix ring} $\mathrm{M}_{n}(R) \defeq R^{n\times n}$ and the \emph{trace map} $\mathrm{tr} \colon \mathrm{M}_{n}(R) \to R, \, m \mapsto \sum_{i=1}^{n} m_{ii}$, which is an $R$-bimodule homomorphism.

\begin{lem}\label{lemma:torsion-free.matrix} Let $R$ be a ring and let~$E$ be an abelian group.  If a homomorphism $\chi \colon R \to E$ is left torsion-free (resp., right torsion-free), then so is $\chi \circ \mathrm{tr} \colon \mathrm{M}_{n}(R) \to E$ for every $n \geq 1$. \end{lem}

\begin{proof} Suppose that $\chi \colon R \to E$ is left torsion-free and let~$I$ be a left ideal in~$M_n(R)$ contained in $\ker( \chi \circ \mathrm{tr})$.  Since the trace map is $R$-linear, it follows that~$\mathrm{tr}(I)$ is a left ideal in~$R$ contained in $\ker \chi$, and therefore $\mathrm{tr}(I) = 0$.  We claim that $I = 0$.  Let $a = (a_{i j}) \in I$, and denoting by $e^{i j} \in \mathrm{M}_n(R)$ the elementary matrix with $(e^{i j})_{i j} = 1$ and $(e^{i j})_{k \ell} = 0$ if $(k, \ell) \ne (i, j)$, we have $\mathrm{tr} ( e^{i j} a ) = \mathrm{tr} ( e^{i j} \sum_{k, \ell} a_{k \ell} e^{k \ell} ) = \mathrm{tr} ( \sum_{\ell} a_{j \ell} e^{i \ell} ) = a_{j i}$.  Since $e^{i j} a \in I$ for all $i, j$ and $\mathrm{tr}(I) = 0$, we conclude $a = 0$ as desired.  The proof for the right torsion-free case is analogous. \end{proof}

\begin{lem}\label{lemma:torsion-free.product} Let $R_{1},\ldots,R_{n}$ be rings and let $E$ be an abelian group.  For any left torsion-free (resp., right torsion-free) homomorphisms $\chi_{i} \colon R_{i} \to E$ $(i \in \{ 1,\ldots,n \})$, the homomorphism \begin{displaymath}
	R_{1} \times \ldots \times R_{n} \to E, \quad (r_{1}, \ldots, r_{n}) \mapsto \textstyle\sum\limits_{i=1}^{n} \chi_{i}(r_{i})
\end{displaymath} is left torsion-free (resp., right torsion-free), too. \end{lem}

\begin{proof} Let $R \defeq R_1 \times \ldots \times R_n$ and let $\chi \colon R \to E$, $(r_1, \dots, r_n) \mapsto \sum_{i=1}^n \chi_i(r_i)$.  Suppose that the $\chi_i \colon R_i \to E$ are left torsion-free and let~$I$ be a left ideal in~$R$ contained in $\ker \chi$.  For $i \in \{ 1, \ldots, n \}$ denote by $\pi_i \colon R \to R_i$ the projection and let $e^i \in R$ be the central idempotent with $(e^i)_i = 1$ and $(e^i)_j = 0$ for $j \ne i$.  Then $\pi_i(e^i I)$ is a left ideal in~$R_i$ contained in $\ker \chi_i$.  It follows that $\pi_i(e^i I) = 0$ and thus $e^i I = 0$ for all $i \in \{ 1, \dots, n \}$, which implies $I = 0$ as desired.  The proof for the right torsion-free case is analogous. \end{proof}

\begin{cor}\label{corollary:torsion-free.semisimple} Every semisimple ring admits a torsion-free homomorphism into $\mathbb Q / \mathbb Z$. \end{cor}

\begin{proof} Thanks to the famous Artin-Wedderburn theorem, any semisimple ring is isomorphic to $M_{n_1}(D_1) \times \ldots \times M_{n_m}(D_m)$ for suitable positive integers $n_1, \ldots, n_m$ and division rings $D_1, \ldots, D_m$.  Thus we may apply Lemma~\ref{lemma:torsion-free.division}, Lemma~\ref{lemma:torsion-free.matrix} and Lemma~\ref{lemma:torsion-free.product}. \end{proof}

We are now ready to establish the first main result of the present note. Our proof utilizes Pontryagin duality for modules, which we recall in the appendix.

\begin{thm}\label{theorem:frobenius.characters} Let $R$ be a left Artinian ring. The following are equivalent: \begin{enumerate}
	\item[$(1)$] $\soc^{\ast}({}_R R)$ embeds into ${}_R (R /\! \rad R)$,
	\item[$(2)$] $R$ admits a finitarily left torsion-free homomorphism into $\mathbb Q / \mathbb Z$,
	\item[$(3)$] $R$ admits a finitarily left torsion-free character.
\end{enumerate} \end{thm}

\begin{proof} $(1) \!\Longrightarrow\! (2)$.  Since $S \defeq R /\! \rad R$ is semisimple there is by Corollary~\ref{corollary:torsion-free.semisimple} a torsion-free homomorphism $\chi \colon S \to \mathbb Q / \mathbb Z$.  By hypothesis we have an embedding $\phi \colon \soc^{\ast}({}_R R) \to {}_R S$.  Now since $\mathbb Q / \mathbb Z$ is divisible, i.e., injective as a $\mathbb Z$-module, there exists a homomorphism $\ol \chi \colon R \to \mathbb Q / \mathbb Z$ such that $\ol \chi \vert_{\soc^{\ast}({}_R R)} = \chi \circ \phi$.  We claim that~$\ol \chi$ is finitarily left torsion-free.  To see this, let~$I$ be any nonzero finite left ideal of~$R$.  Then we find a minimal (nonzero) left ideal~$I_0$ of~$R$ such that $I_0 \subseteq I$.  As~$I_0$ is finite, $I_0 \subseteq \soc^{\ast}({}_R R)$.  It follows that $\phi(I_0)$ is a nonzero submodule of~${}_R S$, i.e., $\phi(I_0)$ is a nonzero left ideal of the ring~$S$.  Since~$\chi$ is left torsion-free, we have that $\phi(I_0) \nsubseteq \ker \chi$ and thus $I_0 \nsubseteq \ker (\chi \circ \phi)$.  By choice of~$\ol \chi$ and~$I_0$, this implies that $I \nsubseteq \ker \ol \chi$.  This shows that $\ker \ol \chi$ contains no nonzero finite left ideal.
	
$(2) \!\Longrightarrow\! (3)$.  Since $\mathbb T \cong \mathbb R / \mathbb Z$, this is obvious.

$(3) \!\Longrightarrow\! (1)$.  Let~$R$ be left Artinian.  By applying the Artin-Wedderburn theorem we find a finite semisimple ring~$E$ and a semisimple ring~$U$ without non-trivial finite left modules, together with a surjective homomorphism $h \colon R \to E \times U$ with $\ker h = \rad R$.  Consider the projection $p \colon E \times U \to E$ and let $h_E \defeq p \circ h$, so that $K \defeq \ker h_E = h^{-1}(U)$.  Since $\ker h = \rad R$ and~$U$ has no non-trivial finite left modules, it is easy to see that $K T = 0$ for every minimal finite left ideal~$T$ of~$R$.  We conclude that $K \soc^{\ast}(R) = 0$.

Now suppose that $\chi \colon R \to \mathbb T$ is a finitarily left torsion-free homomorphism.  For each $a \in A \defeq \soc^{\ast}(R)$ we have just shown that $K \subseteq \ker(a \chi)$, whence there exists a unique $a . \chi \in \smash{\widehat E}$ such that $a . \chi \circ h_E = a \chi$.  Moreover, viewing~$E$ as a right $R$-module, the homomorphism $h_E \colon R_R \to E_R$ induces a homomorphism $\phi \colon {}_R A \to {}_R \smash{\widehat E}$, $a \mapsto a . \chi$.  Furthermore, as~$\chi$ is finitarily left torsion-free we deduce that~$\phi$ is injective: if $a \in A \setminus \{ 0 \}$, then $R a \nsubseteq \ker \chi$, i.e., there exists $r \in R$ such that $1 \ne \chi(r a) = (a . \chi) (h_E(r))$, wherefore $a . \chi \ne 1$.  This shows that ${}_R A$ embeds into ${}_R \smash{\widehat E}$.  Finally, since~$E$ is finite and semisimple, thus Frobenius, we have that ${}_E \smash{\widehat E} \cong {}_E E$ by work of Wood~\cite[Thm.~3.10]{wood1}, and hence ${}_R \smash{\widehat E} \cong {}_R E$.  Thus the composition of embeddings ${}_R A \to {}_R \smash{\widehat E} \to {}_E E \to {}_R (R /\! \rad R)$ provides an embedding of ${}_R A = \soc^{\ast}({}_R R)$ into ${}_R (R /\! \rad R)$, as desired. \end{proof}

Let us also add a direct argument for the implication $(3) \!\Longrightarrow\! (2)$ of Theorem~\ref{theorem:frobenius.characters}.  Suppose for a left Artinian ring~$R$ a finitarily left torsion-free homomorphism $\chi \colon R \to \mathbb{R}/\mathbb Z$ is given.  Since $F \defeq \soc^{\ast}({}_R R)$ is finite, $\chi (F)$ is a finite subgroup of $\mathbb R / \mathbb Z$, thus contained in the torsion subgroup $\mathbb Q / \mathbb Z$.  By divisibility of $\mathbb Q/\mathbb Z$, there exists a homomorphism $\chi^{\ast} \colon R \to \mathbb Q / \mathbb Z$ such that $\chi^{\ast}|_F = \chi|_F$.  In particular, $F \cap (\ker \chi^{\ast}) = F \cap (\ker \chi)$.  In turn, $\chi^{\ast}$ must be finitarily left torsion-free, as every finite left ideal of~$R$ contains a minimal left ideal.

By the method of proof, we have the following.

\begin{cor}\label{corollary:frobenius.characters} Let~$R$ be a left Artinian ring. If $\soc({}_R R) \cong {}_R (R /\! \rad R)$ (in particular, if $R$ is Frobenius), then~$R$ admits a left torsion-free homomorphism into $\mathbb Q / \mathbb Z$. \end{cor}

\section{Dual modules and almost monotheticity}\label{section:almost.monothetic.modules}

This section offers a topological perspective on (finitarily) Frobenius rings, in terms of compact modules arising via Pontryagin duality.  Let us start off with a simple characterization of torsion freeness of characters.  For notation, see the appendix.

\begin{lem}\label{lemma:torsion-free.density} Let~$R$ be a ring and let $\chi \in \smash{\widehat R}$.  The following hold. \begin{enumerate}
	\item[$(1)$] The character $\chi$ is left torsion-free (resp., right torsion-free) if and only if $\chi R$ (resp., $R\chi$) is dense in $\smash{\widehat R}$.
	\item[$(2)$] The character $\chi$ is finitarily left torsion-free (resp., right torsion-free) if and only if $\chi$ is not contained in a finite-index closed proper submodule of $\smash{\widehat R_R}$ (resp., $\smash{{}_R \widehat R}$).
\end{enumerate} \end{lem}
	
\begin{proof} Consider the closed submodule $B \defeq \overline{\chi R} \leq \widehat{R}_R$ and the corresponding left ideal \begin{equation*}\tag{$\ast$}\label{annihilator}
	\Delta (B) = \Delta (\chi R) = \{ x \in R \mid Rx \subseteq \ker \chi \} .
\end{equation*} Then $\chi$ is left torsion-free if and only if $\Delta (B) = \{ 0 \}$, which, by Proposition~\ref{proposition:pontryagin}, is the case if and only if $B = \smash{\widehat{R}}$. This proves~(1). In order to show~(2), we infer from~\eqref{annihilator} that $\chi$ is finitarily left torsion-free if and only if $\Delta (B)$ has no nonzero finite left sub-ideals, which, thanks to Proposition~\ref{proposition:pontryagin} and Lemma~\ref{lemma:pontryagin}, just means that $B$ is not contained in any finite-index proper closed submodules of $\smash{\widehat{R}}_R$. Of course, the latter is equivalent to $\chi$ not being contained in any finite-index proper closed submodules of $\smash{\widehat{R}}_R$, which readily completes the argument. The other cases are proven analogously. \end{proof}

We continue with a useful abstract concept for compact modules.  Given a ring~$R$, a \emph{compact right $R$-module} is a compact abelian group~$X$ together with a \emph{continuous} right $R$-module structure, i.e., such that $X \to X$, $x \mapsto x r$ is continuous for every $r \in R$.

\begin{definition} Let $R$ be a ring. A compact right $R$-module $X_R$ is said to be \emph{monothetic} if there exists $x \in X$ such that $\overline{xR} = X$. A compact right $R$-module $X_R$ is called \emph{almost monothetic} if every finite cover of $X_R$ by closed submodules contains $X$ itself, i.e., for every finite set $\mathcal{M}$ of closed submodules of $X_R$ we have \begin{displaymath}
	X = \bigcup \mathcal{M} \quad \Longrightarrow \quad X \in \mathcal{M} \, .
\end{displaymath} \end{definition}

It is easily seen that both monothetic and almost monothetic compact modules provide a generalization of cyclic finite modules, in the sense that a finite right $R$-module $X_R$ is cyclic if and only if $X_R$ is monothetic, if and only if $X_R$ is almost monothetic. For general compact modules, monotheticity implies almost monotheticity.

The term monotheticity was introduced in topological group theory by van~Dantzig~\cite{vanDantzig}: a topological group is said to be \emph{monothetic} if it contains a dense cyclic subgroup (which clearly implies that the group is abelian). For more details on such groups, the reader is referred to~\cite{HalmosSamelson, TopologicalGroups}. Considering compact abelian groups as compact right $\mathbb{Z}$-modules, our definition of monotheticity above naturally extends van Dantzig's concept to the realm of compact modules over arbitrary rings. Almost monotheticity appears to be the right generalization thereof in the context of MacWilliams' extension property, as will be substantiated by Theorem~\ref{theorem:almost.monothetic} and~Corollary~\ref{corollary:summary}.

Utilizing the following combinatorial Lemma~\ref{lemma:passman.gottlieb}, we will provide a simple characterization of almost monothetic compact modules in Proposition~\ref{proposition:almost.monothetic}.

\begin{lem}[{\cite[Lem.~5.2]{passman}; see also~\cite[Thm.~18]{gottlieb}}]\label{lemma:passman.gottlieb} Let $G$ be an abelian group. If $\mathcal{H}$ is a finite cover of $G$ by subgroups such that $G \ne \bigcup \mathcal{H} \setminus \{ H \}$ for every $H \in \mathcal{H}$, then $G / \bigcap \mathcal{H}$ is finite. \end{lem}

\begin{prop}\label{proposition:almost.monothetic} Let $R$ be a ring. A compact right $R$-module $X_R$ is almost monothetic if and only if every finite cover of $X_R$ by finite-index closed submodules contains $X$ itself. \end{prop}

\begin{proof} The implication ($\Longrightarrow$) is obvious. In order to prove ($\Longleftarrow$), let $\mathcal{M}$ be a finite set of closed submodules of $X_R$ with $X = \bigcup \mathcal{M}$. We wish to show that $X \in \mathcal{M}$. Without loss of generality, we may assume that $X \ne \bigcup \mathcal{M}\setminus \{ M \}$ for every $M \in \mathcal{M}$. Thanks to Lemma~\ref{lemma:passman.gottlieb}, each member of $\mathcal{M}$ then has finite index in $X_R$, whence $X \in \mathcal{M}$ by our hypothesis. \end{proof}

\begin{cor}\label{corollary:almost.monothetic} Let $R$ be a ring and let $X_R$ be a compact right $R$-module. If $X_R$ is not covered by its finite-index closed proper submodules, then $X_R$ is almost monothetic. \end{cor}

We now return to Pontryagin duals of Artinian rings.  The subsequent result characterizes the finitarily Frobenius rings in topological terms, in turn offering an approach to the proof of the general MacWilliams theorem.

\begin{thm}\label{theorem:almost.monothetic} Let $R$ be a left Artinian ring. Then $\soc^{\ast}({}_R R)$ embeds into ${}_R (R /\! \rad R)$ if and only if $\smash{\widehat R}_R$ is almost monothetic. \end{thm}

\begin{proof} ($\Longrightarrow$) This follows from Theorem~\ref{theorem:frobenius.characters} and Lemma~\ref{lemma:torsion-free.density}(2) along with Corollary~\ref{corollary:almost.monothetic}.
	
($\Longleftarrow$) Suppose that $\smash{\widehat R}_R$ is almost monothetic. Since $R$ is left Artinian, the set $\mathcal L$ of all finite simple left ideals of $R$ is finite.  By Proposition~\ref{proposition:pontryagin} and Lemma~\ref{lemma:pontryagin}, the finite set $\mathcal{M} \defeq \{ \Gamma (I) \mid I \in \mathcal{L} \}$ consists of closed proper submodules of $\smash{\widehat R}_R$. As $\smash{\widehat R}_R$ is almost monothetic, there exists $\chi \in \smash{\widehat R}$ with $\chi \notin \bigcup \mathcal{M}$, i.e., $I \nsubseteq \ker \chi$ for every $I \in \mathcal{L}$.  Since every nonzero finite left ideal of $R$ contains a member of $\mathcal{L}$, it follows that $\chi$ is finitarily left torsion-free.  Hence, $\soc^{\ast}({}_RR)$ embeds into ${}_R (R /\! \rad R)$ by Theorem~\ref{theorem:frobenius.characters}. \end{proof}

The following lemma is the main reason for our interest in almost monothetic modules.

\begin{lem}\label{lemma:almost.monothetic} Let $R$ be a ring, $X_R$ and~$Y_R$ be compact right $R$-modules, where~$X_{R}$ is almost monothetic.  Let $f_1, \dots, f_n, g_1, \dots, g_n \colon X_R \to Y_R$ be continuous homomorphisms such that \begin{displaymath}
	\forall \alpha \in \widehat Y \colon \qquad \sum_{i=1}^n \int \alpha (f_{i}(x)) \, d \mu_X(x) = \sum_{i=1}^n \int \alpha (g_{i}(x)) \, d \mu_X(x) \,.
	\end{displaymath} Then, for each $j \in \{1, \dots, n \}$, there exists $k \in \{ 1, \dots, n \}$ such that $\ker \widehat {g_k} \subseteq \ker \smash{\widehat {f_j}}$. \end{lem}

\begin{proof} Since by Theorem~\ref{theorem:bohr} the linear span of $\smash{\widehat Y}$ is dense in $C(Y)$, and by continuity of the map $C(Y) \to \mathbb C$, $h \mapsto \sum_{i=1}^n \int h (f_{i}(x) - g_{i}(x)) \, d \mu_X(x)$, our hypothesis implies that \begin{displaymath}
	\forall h \in C(Y) \colon \qquad \sum_{i=1}^n \int h (f_{i}(x) - g_{i}(x)) \, d \mu_X(x) = 0 \,.
\end{displaymath} Let $j \in \{ 1, \dots, n \}$.  Then $f_{j}(X)$ is contained in $B \defeq \bigcup_{k=1}^{n} g_k(X)$: otherwise, assuming that $f_j(x) \notin B$ for some $x \in X$ and noting that $B$ is closed in $Y$, by Urysohn's lemma we find $h \in C(Y)$ with $h \ge 0$ such that $h|_B \equiv 0$ and $h(f_j(x)) > 0$, which implies that \begin{displaymath}
	0 = \sum_{i=1}^n \int h (f_{i}(x) - g_{i}(x)) \, d \mu_X(x) \ge \int h (f_{j}(x)) \, d \mu_X(x) > 0
\end{displaymath} and thus gives a contradiction. Hence, $X = \bigcup_{k=1}^{n} f^{-1}_{j}(g_{k}(X))$. Since $X$ is almost monothetic, there exists $k \in \{ 1,\ldots,n \}$ such that $X = f^{-1}_{j}(g_{k}(X))$, i.e., $f_{j}(X) \subseteq g_{k}(X)$. We show that $\ker \widehat {g_k} \subseteq \ker \smash{\widehat {f_j}}$.  To this end, let $\kappa \in \smash{\widehat Y}$ with $\kappa \in \ker \widehat {g_k}$, i.e., $\kappa \circ g_k = 1$.  Then \begin{displaymath}
	(\kappa \circ f_j)(X) = \kappa(f_j(X)) \subseteq \kappa(g_k(X)) = (\kappa \circ g_k)(X) = 1 \,,
\end{displaymath} so that $\kappa \in \ker \smash{\widehat {f_j}}$ as desired. \end{proof}

We finish this section with the observation that, by the method of proof of Theorem~\ref{theorem:almost.monothetic}, we have the following.

\begin{cor} Let $R$ be a left Artinian ring. If $\soc({}_RR) \cong {}_R (R / \rad R)$ (in particular, if $R$ is Frobenius), then $\smash{\widehat R}_R$ is monothetic. \end{cor}

\begin{proof} This is an immediate consequence of Corollary~\ref{corollary:frobenius.characters} and Lemma~\ref{lemma:torsion-free.density}(1). \end{proof}

\section{MacWilliams' extension theorem for the Hamming weight}\label{section:macwilliams}

In this section we prove MacWilliams' extension theorem for the Hamming weight on general Frobenius rings.  Let us start off with some basic terminology.  Let~$G$ be an abelian group.  By a \emph{weight} on~$G$ we mean any function from~$G$ to~$\mathbb C$.  Given a weight $w \colon G \to \mathbb C$ and any positive integer~$n$, we denote $w(x) \defeq \sum_{i=1}^{n} w(x_{i})$ for $x \in G^{n}$.  The \emph{Hamming weight} $w_{\mathrm{H}} \colon G \to \mathbb{C}$ is defined by \begin{displaymath}
	w_{\mathrm{H}}(x) \defeq \begin{cases}
		0 & \text{if } x = 0, \\
		1 & \text{otherwise}
	\end{cases} \qquad (x \in G) .
\end{displaymath} The following well-known general character-theoretic observation, noted in~\cite[p.~572, Eq.~(1)]{iovanov}, connects the Hamming weight with the Haar integration on Pontryagin duals. For notation, see the appendix.

\begin{lem}\label{lemma:weights} Let $G$ be an abelian group. For every $x \in G$, \begin{displaymath}
	w_{\mathrm{H}}(x) = 1 - \int \gamma(x) \, d\mu_{\widehat{G}}(\gamma) .
\end{displaymath} \end{lem}

\begin{proof} Noting that $x \ne 0$ if and only if $\eta_G(x) \ne 1$ by Theorem~\ref{theorem:pontryagin}, the result is immediate from Lemma~\ref{lemma:characters.are.ap} applied to the group $\smash{\widehat G}$. \end{proof}

\begin{cor}\label{corollary:weights} Let $G$ be an abelian group and $n \geq 1$. For every $x \in G^{n}$, \begin{displaymath}
	w_{\mathrm{H}}(x) = n - \sum_{i=1}^{n} \int \gamma(x_{i}) \, d\mu_{\widehat{G}}(\gamma) .
\end{displaymath} \end{cor}

We proceed to rings.  Our main focus will be on the MacWilliams property. Let $R$ be a ring and consider its \emph{group of units} \begin{displaymath}
U(R) \defeq \{ u \in R \mid \exists v \in R \colon \, uv = vu = 1 \} .
\end{displaymath} Given $n \geq 1$, $\sigma \in S_{n}$ and $u \in U(R)^{n}$, we consider the module automorphisms \begin{gather*}
\Phi_{\sigma,u} \colon {}_R R^{n} \to {}_R R^{n}, \quad x \mapsto (x_{\sigma 1}u_{1},\ldots,x_{\sigma n}u_{n}) , \\
\Psi_{\sigma,u} \colon R^{n}_R \to R^{n}_R, \quad x \mapsto (u_{1}x_{\sigma 1},\ldots,u_{n}x_{\sigma n}) ,
\end{gather*} and note that $w_{\mathrm{H}}(\Phi_{\sigma,u}(x)) = w_{\mathrm{H}}(\Psi_{\sigma,u}(x)) = w_{\mathrm{H}}(x)$ for all $x \in R^{n}$.

\begin{definition} A ring~$R$ is called \emph{left MacWilliams} if, for every integer $n \ge 1$ and any homomorphism $\phi \colon {}_R M \to {}_R N$ between submodules $M, N$ of ${}_R R^{n}$ with $w_{\mathrm{H}} (\phi(x)) = w_{\mathrm{H}}(x)$ for all $x \in M$, there exist $\sigma \in S_{n}$ and $u \in U(R)^{n}$ with $\phi = \Phi_{\sigma,u}\vert_{M}^{N}$.  Analogously, a ring~$R$ will be called \emph{right MacWilliams} if, for every integer $n \ge 1$ and any homomorphism $\phi \colon M_R \to N_R$ between submodules $M, N$ of $R^{n}_R$ with $w_{\mathrm{H}} (\phi(x)) = w_{\mathrm{H}}(x)$ for all $x \in M$, there exist $\sigma \in S_{n}$ and $u \in U(R)^{n}$ with $\phi = \Psi_{\sigma,u}\vert_{M}^{N}$. \end{definition}

Our goal is to establish a link between the MacWilliams property and the finitary Frobenius property.  The next two lemmata together constitute a key observation. The arguments are reminiscent of Iovanov's work~\cite[Sec.~4.1]{iovanov}.

\begin{lem}\label{lemma:main} Let $R$ be a ring such that $\smash{\widehat R}_R$ is almost monothetic. Let $n\geq 1$, let $M$ be a left $R$-module and let $\phi, \psi \colon {}_RM \to {}_RR^{n}$ be homomorphisms with $w_{\mathrm{H}}(\phi(x)) = w_{\mathrm{H}}(\psi(x))$ for all $x \in M$.  Then, \begin{displaymath}
	\forall j \in \{ 1, \dots, n \} \ \exists k \in \{ 1,\ldots,n \} \colon \qquad \ker \psi_k \subseteq \ker \phi_j .
\end{displaymath} \end{lem}

\begin{proof} In light of Corollary~\ref{corollary:weights}, our assumption means that \begin{displaymath}
	\forall x \in M \colon \qquad \sum_{i=1}^n \int \gamma(\phi_i(x)) \, d \mu_{\widehat R}(\gamma) =\sum_{i=1}^n \int \gamma(\psi_i(x)) \, d \mu_{\widehat R}(\gamma) \,,
\end{displaymath} or equivalently, \begin{displaymath}
	\forall x \in M \colon \qquad \sum_{i=1}^n \int \eta_M(x)(\widehat{\phi_i}(\gamma)) \, d \mu_{\widehat R}(\gamma) =\sum_{i=1}^n \int \eta_M(x)(\widehat{\psi_i}(\gamma))) \, d \mu_{\widehat R}(\gamma) \,.
\end{displaymath} The result then follows by applying Lemma~\ref{lemma:almost.monothetic} to $X_R = \smash{\widehat R}_R$ and $Y_R = \smash{\widehat M}_R$, together with the fact that $\eta_{M} \colon M \to \smash{\widehat{Y}}$ is an isomorphism by Pontryagin duality (Theorem~\ref{theorem:pontryagin}). \end{proof}

The conclusion of Lemma~\ref{lemma:main} will now be adapted as follows.

\begin{lem}\label{lemma:main.too} Let~$R$ be a ring such that $\smash{\widehat R}_R$ is almost monothetic. Let $n \ge 1$, let~$M$ be a left $R$-module and let $\phi, \psi \colon {}_R M \to {}_R R^{n}$ be homomorphisms such that $w_{\mathrm{H}}(\phi(x)) = w_{\mathrm{H}}(\psi(x))$ for all $x \in X$. Then there exist $j, k \in \{ 1, \dots, n \}$ such that $\ker \phi_j = \ker \psi_k$. \end{lem}

\begin{proof} Let $j_0 \defeq 1$.  By Lemma~\ref{lemma:main}, there are $k_0, \ldots, k_{n-1}, j_1, \ldots, j_n \in \{ 1, \dots, n \}$ with \begin{displaymath}
	\ker \phi_{j_0} \supseteq \ker \psi_{k_0} \supseteq \ker \phi_{j_1} \supseteq \ker \psi_{k_1} \supseteq \ldots \supseteq \ker \psi_{k_{n-1}} \supseteq \ker \phi_{j_{n}} .
\end{displaymath} Clearly, $j_{k} = j_{\ell}$ for some $k,\ell \in \{ 0,\ldots,n \}$ with $k < \ell$, and thus $\ker \phi_{j_{k}} = \ker \psi_{j_{k}}$. \end{proof}

We are ready to prove the generalized MacWilliams extension theorem.

\begin{prop}\label{proposition:macwilliams} Every left Artinian, left pseudo-injective ring~$R$ such that $\soc^*({}_R R)$ embeds into ${}_R (R /\! \rad R)$ is left MacWilliams. \end{prop}

\begin{proof} By virtue of Theorem~\ref{theorem:almost.monothetic} we have that $\smash{\widehat R}_{R}$ is almost monothetic. Given $n\geq 1$, any left $R$-module $M$ and homomorphisms $\phi,\psi \colon {}_RM \to {}_RR^{n}$ with $w_{\mathrm{H}}(\phi(x)) = w_{\mathrm{H}}(\psi(x))$ for all $x \in M$, we need to show that there exist $\sigma \in S_{n}$ and $u \in U(R)^{n}$ such that $\psi = \Phi_{\sigma,u} \circ \phi$. Our proof proceeds by induction on $n \geq 1$. For the induction base, let $n = 1$. Since $\smash{\widehat{R}_{R}}$ is almost monothetic, Lemma~\ref{lemma:main.too} implies that $\ker \phi = \ker \psi$. Thanks to Lemma~\ref{lemma:bass}, there exists $u \in U(R)$ such that $\psi(x) = \phi(x)u$ for all $x \in M$, i.e., $\psi = \Phi_{\mathrm{id},u} \circ \phi$. For the inductive step, suppose that the statement is true for some $n \geq 1$. As $\smash{\widehat{R}_{R}}$ is almost monothetic, Lemma~\ref{lemma:main.too} implies that there exist $j, k \in \{ 1, \dots, n+1 \}$ such that $\ker \phi_j = \ker \psi_k$. Using Lemma~\ref{lemma:bass} again, we find $u \in U(R)$ with $\psi_k(x) = \phi_j(x) u$ for all $x \in M$. For every $x \in M$, \begin{displaymath}
	\sum_{i=1}^{n+1} w_{\mathrm{H}}(\phi_i(x)) = \sum_{i=1}^{n+1} w_{\mathrm{H}}(\psi_i(x))
\end{displaymath} by assumption and $w_{\mathrm{H}}(\phi_j(x)) = w_{\mathrm{H}}(\phi_j(x) u) = w_{\mathrm{H}}(\psi_k(x))$, which implies that \begin{displaymath}
	\sum_{i=1, \, i \ne j}^{n+1} w_{\mathrm{H}}(\phi_i(x)) = \sum_{i=1, \, i \ne k}^{n+1} w_{\mathrm{H}}(\psi_i(x)) .
\end{displaymath} Appealing to the induction hypothesis and the fact that $\psi_k(x) = \phi_j(x) u$ for all $x \in M$, we find $\sigma \in S_{n+1}$ and $u \in U(R)^{n+1}$ so that $\psi = \Phi_{\sigma,u} \circ \phi$. \end{proof}

Now we present the main result, which characterizes the Artinian rings satisfying the MacWilliams extension property.  In addition to the results of the previous sections, we use a strategy developed by Dinh, L\'opez-Permouth~\cite{DinhLopezB} and Wood~\cite{wood2}.

\begin{thm}\label{theorem:main} A left Artinian ring~$R$ is left MacWilliams if and only if it is left pseudo-injective and $\soc^*(R)$ embeds into ${}_R (R /\! \rad R)$. \end{thm}

\begin{proof} Every left Artinian, left pseudo-injective ring~$R$ with an embedding of $\soc^*(R)$ into ${}_R (R /\! \rad R)$ is left MacWilliams by Proposition~\ref{proposition:macwilliams}.  For the converse, suppose the ring~$R$ to be left Artinian and left MacWilliams.  From the extension property for codes of length~$1$ we readily infer that~$R$ is left pseudo-injective.  It remains to prove that $\mathrm{soc}^{\ast}({}_{R}R)$ embeds into ${}_{R}(R /\! \rad R)$.  To this end, let the Artin-Wedderburn decomposition of $S \defeq R /\! \rad R$ be \[ S \cong \textstyle\bigoplus\limits_{i=1}^m {\mathrm M}_{\mu_i}(D_i) \] for some positive integers $\mu_1, \ldots, \mu_m$ and division rings~$D_i$, so there is a basic set $e_1, \dots, e_m$ of idempotents in~$R$ such that ${}_R S \cong \bigoplus_{i=1}^m (S \ol e_{i})^{\mu_i}$ and each simple left $R$-module is isomorphic to some $S \ol e_i$.  Also, there are natural numbers $\nu_1, \ldots, \nu_m$ with $\soc({}_R R) \cong \bigoplus_{i=1}^{m} (S \ol e_{i})^{\nu_i}$.  Without loss of generality, we may assume that the finite modules $S \ol e_i$ are precisely $S \ol e_1, \dots, S \ol e_{\ell}$ for some $\ell \le m$.  We conclude that $\soc^{\ast}({}_R R)$ embeds into ${}_R S$ if and only if $\nu_i \le \mu_i$ for every $i \in \{ 1, \ldots, \ell \}$.  Now assuming for contradiction that $\soc^{\ast}({}_R R)$ does not embed into ${}_R S$, there exists some~$j \in \{ 1, \ldots, \ell \}$ such that $\nu_j > \mu_j$. 

As ${}_R S \ol e_j$ is isomorphic to the pull-back to~$R$ of the standard column module over ${\mathrm M}_{\mu_j}(D_j)$, we may assume that $\soc^{\ast}({}_R R)$ contains a matrix module $\smash{A \defeq D_j^{\mu_j \times \nu_j}}$ over ${\mathrm M}_{\mu_j}(D_j)$.  Wood's result~\cite[Thm.~4.1]{wood2} states the existence of left submodules $C_+, C_-$ of~$A^n$ for some positive integer~$n$ and a Hamming weight preserving isomorphism $f \colon C_+ \to C_-$ such that $C_+$ has an identically zero component while $C_-$ does not.  Of course, $C_+, C_-$ are submodules of~${}_R R^n$ and the isomorphism $f \colon C_+ \to C_-$ cannot be extended to a monomial transformation, contradicting the left MacWilliams property. \end{proof}

We conclude this note with multiple characterizations of the MacWilliams extension property for quasi-Frobenius rings. \pagebreak

\begin{cor}\label{corollary:summary} Let~$R$ be a quasi-Frobenius ring.  The following are equivalent: \begin{enumerate}
	\item[$(1)$] $R$ is finitarily Frobenius,
	\item[$(2)$] $R$ admits a finitarily left torsion-free character,
	\item[$(3)$] $\smash{\widehat R}_R$ is almost monothetic,
	\item[$(4)$] $R$ is left MacWilliams.
\end{enumerate} Also, each of (2) and (4) may be exchanged by its right version, and (3) by its left version. \end{cor}

\begin{proof} We have $(1) \!\Longleftrightarrow\! (2)$ by Theorem~\ref{theorem:frobenius.characters}, $(1) \!\Longleftrightarrow\! (3)$ is due to Theorem~\ref{theorem:almost.monothetic}, and, noting that~$R$ is pseudo-injective, $(1) \!\Longleftrightarrow\! (4)$ according to Theorem~\ref{theorem:main}.  The last statement follows from the symmetry of (1), see Proposition~\ref{proposition:finitary.frobenius}. \end{proof}

Finally, let us state the following characterization of Artinian rings satisfying the MacWilliams property on both sides.

\begin{cor} An Artinian ring is finitarily Frobenius iff it is left and right MacWilliams. \end{cor}

\begin{proof} This follows at once from Proposition~\ref{prop:quasi-Frobenius.pseudo-injective} and Theorem~\ref{theorem:main}. \end{proof} 

A natural open question now arises in light of the results presented in this paper.

\begin{quest} Suppose that~$R$ is an Artinian ring which is left MacWilliams.  Does it follow that~$R$ is quasi-Frobenius, i.e., is the ring~$R$ also right MacWilliams? \end{quest}

\appendix

\section{Abstract harmonic analysis}\label{section:pontryagin}

In this appendix, we shortly recollect some basics of abstract harmonic analysis: Pontryagin duality, Bohr approximation, and Haar integration.  For more on this, we refer to~\cite{DeitmarEchterhoff}.

Consider the circle group~$\mathbb{T} \defeq \{ z \in \mathbb C \mid \vert z \vert = 1 \}$, which is isomorphic to the quotient~$\mathbb R / \mathbb Z$.  Notice that the group~$\mathbb T$ is written multiplicatively, whereas the group $\mathbb R / \mathbb Z$ additively. Let~$G$ be a locally compact abelian group. We denote by $\smash{\widehat G}$ the \emph{dual group} of~$G$, i.e., the topological group of all continuous homomorphisms from~$G$ into~$\mathbb{T}$ endowed with the compact-open topology, which constitutes a locally compact abelian group itself. As usual, the elements of~$\smash{\widehat{G}}$ are called \emph{characters} on $G$.

\begin{thm}[Pontryagin]\label{theorem:pontryagin} If~$G$ is a locally compact abelian group, then the map \begin{displaymath}
	\eta_G \colon G \to \smash{\widehat{\widehat{G}}}, \quad g \mapsto (\gamma \mapsto \gamma (g))
	\end{displaymath} establishes an isomorphism of topological groups. \end{thm} 

If $\phi \colon G \to H$ is a continuous homomorphism between locally compact abelian groups, then the map $\smash{\widehat \phi \colon \widehat H \to \widehat G}$, $\beta \mapsto \beta \circ \phi$ defines again a continuous homomorphism.  The assignment $\smash{G \mapsto \widehat G}$ becomes this way a functor from the category of locally compact abelian groups (with continuous homomorphisms as morphisms) into itself, and Theorem~\ref{theorem:pontryagin} actually provides a natural equivalence between a locally compact abelian group and its bidual.

We are going to recollect some bits about annihilating subgroups. Let $G$ be a locally compact abelian group. For subsets $A \subseteq G$ and $B \subseteq \smash{\widehat{G}}$, let us define \begin{align*}
&\Gamma (A) \defeq \{ \gamma \in \smash{\widehat{G}} \mid \gamma|_A \equiv 1 \} , &\Delta (B) \defeq \bigcap \{ \ker \gamma \mid \gamma \in B \} ,
\end{align*} noting that $\Gamma (A)$ is a closed subgroup of~$\smash{\widehat{G}}$, while $\Delta (B)$ constitutes a closed subgroup of~$G$.

\begin{prop}\label{proposition:pontryagin} Let $G$ be a locally compact abelian group. Then $\Gamma$ and $\Delta$ constitute mutually inverse order-reversing bijections between the closed subgroups of $G$ and $\smash{\widehat{G}}$. Moreover, for any closed subgroups $H \leq G$ and $K \leq \smash{\widehat{G}}$, \begin{displaymath}
	\Gamma (H) \cong \smash{\widehat{G/H}} , \qquad \Delta (K) \cong \smash{\widehat{\widehat{G}/K}} ,
	\end{displaymath} wherefore $\Gamma (H)$ is finite if and only if $H$ has finite index in $G$, and $\Delta (K)$ is finite if and only if $K$ has finite index in $\smash{\widehat{G}}$. \end{prop}

Let us briefly turn to (discrete) rings and compact modules arising via Pontryagin duality. For a left (resp., right) $R$-module~$M$, the compact dual group $\smash{\widehat M} \defeq \Hom(M, \mathbb T)$ admits a continuous right (resp., left) $R$-module structure given by \begin{displaymath}
(\chi r)(x) \defeq \chi (rx), \qquad ( \text{resp.,} \ (r \chi)(x) \defeq \chi (xr)) \qquad \big(\chi \in \widehat{M}, \, r \in R, \, x \in M \big) .
\end{displaymath} Moreover, if $\phi \colon M \to N$ is a homomorphism between left (resp., right) $R$-modules $M$ and $N$, then the continuous homomorphism $\smash{\widehat{\phi} \colon \widehat{N} \to \widehat{M}}$, $\chi \mapsto \chi \circ \phi$ is in line with the right (resp., left) $R$-module structure of $\smash{\widehat N}$ and $\smash{\widehat M}$. In particular, this construction applies to the $R$-bimodule $M = {}_R R_R$, which therefore gives rise to the compact $R$-bimodule ${}_R  \smash{\widehat R}_R$.

\begin{lem}\label{lemma:pontryagin} Let $R$ be a ring. A subgroup $A$ of a left (resp., right) $R$-module~$M$ is a submodule of $M$ if and only if $\Gamma (A)$ is a submodule of the right (resp., left) $R$-module~$\smash{\widehat{M}}$. In particular, a subgroup $I \leq R$ is a left (resp., right) ideal in $R$ if and only if $\Gamma (I)$ is a submodule of $\smash{\widehat R}_R$ (resp., $\smash{{}_R \widehat R}$). \end{lem}

Next we recall the Bohr approximation theorem. Given a compact (Hausdorff) space $X$, let us consider the commutative $C^{\ast}$-algebra $C(X)$ of all continuous complex-valued functions on $X$, equipped with the obvious point-wise operations and the supremum norm.

\begin{thm}[Bohr approximation theorem]\label{theorem:bohr} Let $G$ a compact abelian group. Then $\smash{\widehat G}$ generates a dense linear subspace of $C(G)$. \end{thm}

\begin{proof} By Theorem~\ref{theorem:pontryagin}, $\smash{\widehat G}$ separates the points of $G$, whence the Stone-Weierstrass theorem asserts that the $\ast$-algebra $A$ generated by $\smash{\widehat G}$ is dense in $C(G)$. But $A$ coincides with the linear subspace generated by $\smash{\widehat G}$, simply because $\smash{\widehat G}$ is closed under point-wise complex conjugation and multiplication. Hence, the theorem follows. \end{proof}

We conclude this appendix with a simple useful fact about continuous characters on compact abelian groups.  Recall that, for a compact group~$G$, there exists a unique left-invariant Radon probability measure $\mu_{G}$ on $G$, which is then necessarily right-invariant, inversion-invariant, and \emph{strictly positive}, i.e., $\mu_{G}(U) > 0$ for any non-empty open subset $U \subseteq G$.

\begin{lem}\label{lemma:characters.are.ap} Let $G$ be a compact abelian group. For every $\chi \in \smash{\widehat G}$, \begin{displaymath}
	\int \chi(x) \, d\mu_{G}(x) = \begin{cases}
	1 & \text{if } \chi = 1 , \\
	0 & \text{otherwise} .
	\end{cases}
	\end{displaymath} \end{lem}

\begin{proof} Clearly, $\int 1 \, d\mu_{G}(x) = 1$. Suppose that $\chi \ne 1$. If $g \in G$ such that $\chi (g) \ne 1$, then \begin{displaymath}
	\int \chi(x) \, d\mu_{G}(x) = \int \chi(gx) \, d\mu_{G}(x) = \int \chi (g) \chi(x) \, d\mu_{G}(x) = \chi(g) \int \chi(x) \, d\mu_{G}(x) ,
	\end{displaymath} which implies that $\int \chi(x) \, d\mu_{G}(x) = 0$. This completes the proof. \end{proof}

\section*{Acknowledgments}

The authors owe special thanks to Tom Hanika for providing a creative and cheerful atmosphere during the final stage of this work, as well as to the referee for a very careful reading of this manuscript and valuable suggestions to improve it.


\end{document}